\documentclass{amsart}
 \usepackage{amssymb, latexsym}
\usepackage{url}

\DeclareMathOperator{\CM}{\mathrm CM}

\DeclareMathOperator{\GL}{\mathrm GL}

\DeclareMathOperator{\norm}{\mathrm norm}
\DeclareMathOperator{\Norm}{\mathrm Norm}

\DeclareMathOperator{\rank}{\mathrm rank}
\DeclareMathOperator{\Reg}{\mathrm Reg}

\DeclareMathOperator{\Tor}{\mathrm Tor}
\DeclareMathOperator{\Trace}{\mathrm Trace}

\begin{document}
 \bibliographystyle{plain}

 \newtheorem{theorem}{Theorem}[section]
 \newtheorem{lemma}{Lemma}[section]
 \newtheorem{corollary}{Corollary}[section]
 \newtheorem{conjecture}{Conjecture}
 \newtheorem{proposition}{Proposition}[section]
 \newtheorem{definition}{Definition}
 \newcommand{\mc}{\mathcal}
 \newcommand{\A}{\mc A}
 \newcommand{\B}{\mc B}
 \newcommand{\eC}{\mc C}
 \newcommand{\D}{\mc D}
 \newcommand{\E}{\mc E}
 \newcommand{\F}{\mc F }
 \newcommand{\G}{\mc G}
 \newcommand{\hH}{\mc H}
 \newcommand{\I}{\mc I}
 \newcommand{\J}{\mc J}
 \newcommand{\eL}{\mc L}
 \newcommand{\M}{\mc M}
 \newcommand{\eN}{\mc N}
 \newcommand{\pP}{\mc P}
 \newcommand{\qq}{\mc Q}
 \newcommand{\eS}{\mc S}
 \newcommand{\eU}{\mc U}
 \newcommand{\V}{\mc V}
 \newcommand{\X}{\mc X}
 \newcommand{\Y}{\mc Y}
 \newcommand{\C}{\mathbb{C}}
 \newcommand{\R}{\mathbb{R}}
 \newcommand{\N}{\mathbb{N}}
 \newcommand{\bP}{\mathbb P}
 \newcommand{\Q}{\mathbb{Q}}
 \newcommand{\T}{\mathbb{T}}
 \newcommand{\Z}{\mathbb{Z}}
 \newcommand{\fA}{\mathfrak A}
 \newcommand{\fB}{\mathfrak B}
 \newcommand{\fD}{\mathfrak D}
 \newcommand{\fd}{\mathfrak d}
 \newcommand{\fE}{\mathfrak E}
 \newcommand{\ff}{\mathfrak F}
 \newcommand{\fI}{\mathfrak I}
 \newcommand{\ffi}{\mathfrak i}
 \newcommand{\fJ}{\mathfrak J}
 \newcommand{\fj}{\mathfrak j}
 \newcommand{\fP}{\mathfrak P}
 \newcommand{\fp}{\mathfrak p}
 \newcommand{\fQ}{\mathfrak Q}
 \newcommand{\fq}{\mathfrak q}
 \newcommand{\fU}{\mathfrak U}
 \newcommand{\fu}{\mathfrak u}
 \newcommand{\fV}{\mathfrak V}
 \newcommand{\fv}{\mathfrak v}
 \newcommand{\fb}{f_{\beta}}
 \newcommand{\fg}{f_{\gamma}}
 \newcommand{\gb}{g_{\beta}}
 \newcommand{\vphi}{\varphi}
 \newcommand{\vep}{\varepsilon}
 \newcommand{\bo}{\boldsymbol 0}
 \newcommand{\bone}{\boldsymbol 1}
 \newcommand{\ba}{\boldsymbol a}
 \newcommand{\bb}{\boldsymbol b}
 \newcommand{\bc}{\boldsymbol c}
 \newcommand{\be}{\boldsymbol e}
 \newcommand{\bk}{\boldsymbol k}
 \newcommand{\bell}{\boldsymbol \ell}
 \newcommand{\bm}{\boldsymbol m}
 \newcommand{\bn}{\boldsymbol n}
 \newcommand{\balpha}{\boldsymbol \alpha}
 \newcommand{\bgamma}{\boldsymbol \gamma}
 \newcommand{\bt}{\boldsymbol t}
 \newcommand{\bu}{\boldsymbol u}
 \newcommand{\bv}{\boldsymbol v}
 \newcommand{\bx}{\boldsymbol x}
 \newcommand{\bX}{\boldsymbol X}
 \newcommand{\bwy}{\boldsymbol y}
 \newcommand{\Bbeta}{\boldsymbol \beta}
 \newcommand{\bxi}{\boldsymbol \xi}
 \newcommand{\bbeta}{\boldsymbol \eta}
 \newcommand{\bw}{\boldsymbol w}
 \newcommand{\bz}{\boldsymbol z}
 \newcommand{\bzeta}{\boldsymbol \zeta}
 \newcommand{\hmu}{\widehat \mu}
 \newcommand{\oK}{\overline{K}}
 \newcommand{\oKt}{\overline{K}^{\times}}
 \newcommand{\oQ}{\overline{\Q}}
 \newcommand{\oq}{\oQ^{\times}}
 \newcommand{\oQt}{\oQ^{\times}/\Tor\bigl(\oQ^{\times}\bigr)}
 \newcommand{\ot}{\Tor\bigl(\oQ^{\times}\bigr)}
 \newcommand{\h}{\frac12}
 \newcommand{\hh}{\tfrac12}
 \newcommand{\dx}{\text{\rm d}x}
 \newcommand{\dbx}{\text{\rm d}\bx}
 \newcommand{\dy}{\text{\rm d}y}
 \newcommand{\dmu}{\text{\rm d}\mu}
 \newcommand{\dnu}{\text{\rm d}\nu}
 \newcommand{\dla}{\text{\rm d}\lambda}
 \newcommand{\dlav}{\text{\rm d}\lambda_v}
 \newcommand{\trho}{\widetilde{\rho}}
 \newcommand{\dtrho}{\text{\rm d}\widetilde{\rho}}
 \newcommand{\drho}{\text{\rm d}\rho}
  
\title[Lower Bounds for Regulators]{Lower Bounds for Regulators of number fields in terms of their discriminants}
\author{Shabnam~Akhtari}
\author{Jeffrey~D.~Vaaler}

\address{Department of Mathematics, University of Oregon, Eugene, Oregon 97403 USA}
\email{akhtari@uoregon.edu}

\address{Department of Mathematics, University of Texas, Austin, Texas 78712 USA}
\email{vaaler@math.utexas.edu}
\numberwithin{equation}{section}

\begin{abstract}  

We prove inequalities that compare the regulator of a number field with the absolute value of its discriminant. We refine the ideas in Silverman's work \cite{silverman1984} where such general inequalities are first proven. In order to prove our main theorems, we combine these refinements with the  authors' previous results on bounding the product of heights of relative units in a number field extension.
\end{abstract}

\maketitle

\section{Introduction}

Let $k$ be an algebraic  number field of degree $d \geq 2$ with regulator $\Reg(k)$,  discriminant 
$\Delta_k$, and absolute discriminant $D_k = \left| \Delta_k \right|$. We denote the ring of algebraic integers in $k$  by $O_k$  and we write $r(k)$ for 
the rank of the unit group $O_k^{\times}$. For every number field with large enough absolute discriminant, an interesting lower bound for $\Reg(k)$ 
in terms of $D_k$ has been established by Silverman in \cite{silverman1984} (see also 
\cite{Cus, pohst1977, remak1952} for such lower bounds in  special cases).
In \cite{friedman1989} Friedman has shown that $\Reg(k)$ takes its minimum value at the unique 
number field $k_0$ having degree $6$ over $\Q$, and having discriminant equal to $-10051$.  By Friedman's result we have
\begin{equation}\label{intro61}
0.2052 \dots = \Reg(k_0) \le \Reg(k)
\end{equation}
for all algebraic number fields $k$. 
Following \cite{silverman1984} we define
\begin{equation}\label{extra15}
\rho(k) = \max\big\{r\bigl(k^{\prime}\bigr) : \text{$k^{\prime} \subseteq k$ and $k^{\prime} \not= k$}\big\}.
\end{equation}
In \cite{silverman1984} Silverman shows that 
\begin{equation}\label{candgamma}
 c_{d} \left(\log \gamma_{d} D_{k}\right)^{\left(r(k) - \rho(k)\right)} < \Reg(k),
\end{equation}
with 
\begin{equation}\label{Silvermanvalues}
c_{d} = 2^{-4d^2} \, \, \qquad \textrm{and} \, \,  \quad\gamma_d = d^{-d^{\log_{2}8d}},
\end{equation}
 and it is understood that $1 < \gamma_d D_k$.

This lower bound is improved in \cite{friedman1989}, where Friedman shows there are computable, positive, absolute constants $C_4$ and $C_5$ 
such that the inequality \eqref{candgamma} holds with
\begin{equation}\label{friedmanremark}
c_d = C_4 d^{-2r+\rho-\frac{1}{2}} (C_5 \log d)^{-3\rho}\, \, \qquad \textrm{and} \, \,  \quad \gamma_d = d^{-d}
\end{equation}
(see the remark after the proof of Theorem C on page 617 of \cite{friedman1989}).

In order to sharpen the values of $c_{d}$ and $\gamma_d$ in the inequalities \eqref{candgamma} that are given in \eqref{Silvermanvalues}  and \eqref{friedmanremark},
we use our results  in \cite{akhtari2021, akhtari2016} that bound the regulators and relative regulators of an extension of number fields  by heights  of units and relative units in the number field extension. 
First we recall that  $\rho(k) = r(k)$ if and only if $k$ is a $\CM$-field (see   
\cite[Corollary 1 to Proposition 3.20]{narkiewicz2010}). If $k$ is a $\CM$-field, then the absolute discriminant of $k$ will not appear in the lower bound 
in \eqref{candgamma}, and in this case the  inequality \eqref{intro61} provides a sharp lower bound. For this reason, in our main theorems  we will 
assume that the number field $k$ is not a $\CM$-field. Another simple case is when $k$ is a totally real quadratic number field. In this case $r(k) = 1$ and $\rho(k) = 0$, and it can be easily seen that 
$$
\tfrac{1}{2} \log \frac{D_{k}}{4} \leq \Reg(k).
$$
So we may assume $d \geq 3$ if need be.
In Theorem \ref{Sil-rho} we will show that one may take $\gamma_d = d^{-d}$,
and in Theorem \ref{Sil-alpha} we will show that one may take $\gamma_d = d^{-\frac{d^{\log_{2} d}}{2}}$.
Both theorems provide explicit values for $c_{d}$ that are  larger than $2^{-4d^2}$.  For clarity  and since  different general strategies are 
used in the proofs,
we state these two theorems separately. 

\begin{theorem}\label{Sil-rho}
Let $k$ be a number field of degree $d \geq 3$ that is not a $\CM$-field, with the unit rank $r = r(k)$ and absolute discriminant $D_{k}$. Let 
$\gamma_{d} = d^{-d}$ and assume that
\begin{equation*}\label{intro65}
1 < \gamma_d D_k.
\end{equation*}
Then we have
\begin{equation}\label{firstDiscineq}
\frac{(2r)!}{(r!)^3}\biggl(\frac{\log \log d}{2 \log d}\biggr)^{3 \rho(k)} \biggl(\frac{\log \left(\gamma_{d} D_k \right)}{4 d}\biggr)^{r - \rho(k)} 
		\le   \Reg(k).
\end{equation}
\end{theorem}

In the proof of Theorem \ref{Sil-rho},  assuming  the truth of Lehmer's conjecture, one can conclude that
\begin{equation*}
\mathfrak{c}^{\rho(k)} \frac{(2r)!}{2^r\, (r!)^3} \biggl(\frac{\log \left(\gamma_{d} D_k \right)}{2 d}\biggr)^{r - \rho(k)} 
		\le   \Reg(k),
\end{equation*}
where $\mathfrak{c}$ is an absolute positive constant.
By appealing to a result of Amoroso and David \cite{AmDa}, which gives a lower bound for the product of heights of algebraic numbers, we  may proceed with the proof of  Theorem \ref{Sil-rho} in Section \ref{therhomethod} to obtain an 
inequality between the regulator and the absolute discriminant that is sharper than \eqref{firstDiscineq}  in terms of the degree of the number field. We obtain
\begin{equation}\label{amorosodavidapplication}
\mathfrak{c}_{0}\, \frac{(2r)!}{(r!)^3} \, \frac{d^{\rho(k)-1}} {(1 + \log d)^{\rho(k) \kappa}} \biggl(\frac{\log \left(\gamma_{d} D_k \right)}{4d}\biggr)^{r - \rho(k)} \le  \Reg(k),
\end{equation}
where $\mathfrak{c}_{0}$ and $\kappa$ depend only on $\rho(k)$ (see the remark at the end of Section \ref{therhomethod} for an explicit version deduced from \cite{AmoV}).


As it is expected that the values obtained for $c_d$ could be improved, we explore two different approaches in our proofs. Our next result is similar to Theorem \ref{Sil-rho}, but is proven using a significantly different strategy which might be useful in some future research. 
\begin{theorem}\label{Sil-alpha}
Let $k$ be a number field of degree $d \geq 3$ that is not a $\CM$-field, with the unit rank $r$ and absolute discriminant $D_{k}$. Let 
$\gamma_d = d^{-\frac{d^{\log_{2} d}}{2}}$ and assume that
\begin{equation*}\label{intro65.5}
1 < \gamma_d D_k.
\end{equation*}
Then we have
\begin{equation}\label{ineq-alpha}
\frac{0.2}{r !} \, \biggl(  \frac{2 d \, {\log \left(\gamma_{d}  D_k \right)}}{(d-2)\, d^{\log_{2}d}}
\biggr)^{r - \rho(k)} 
		\le   \Reg(k).
\end{equation}
\end{theorem}

We recall that  $r(k) + 1$ is the number of archimedean places of $k$, and  therefore $d - 2 \le 2r(k) < 2d$.
Thus in Theorems \ref{Sil-rho} and \ref{Sil-alpha} we may express explicit  values for the constant $c_{d}$ in \eqref{candgamma} in 
terms of $d$ only. In order to compare the values of $c_d$ given in Theorems \ref{Sil-rho} and \ref{Sil-alpha} with that in \eqref{friedmanremark}, we may use 
Stirling's formula
$$
\sqrt{2\pi}\,  n^{n+\frac{1}{2}} e^{-n} \leq n! \leq e\,   n^{n+\frac{1}{2}} e^{-n},
$$
where $n$ is any positive integer.

For the lower bound given in  \eqref{firstDiscineq} for  $\Reg(k)$, we have
\begin{eqnarray*}
&&\frac{(2r)!}{(r!)^3}\biggl(\frac{\log \log d}{2 \log d}\biggr)^{3 \rho(k)} \biggl(\frac{\log \left(\gamma_{d} D_k \right)}{4 d}\biggr)^{r - \rho(k)} \geq \\
&&\frac{ \sqrt{2\pi}\,  (2r)^{2r+\frac{1}{2}} e^{-2r}} {e^3  r^{3r+\frac{3}{2}} e^{-3r}}    \biggl(\frac{\log \log d}{2 \log d}\biggr)^{3 \rho(k)} \biggl(\frac{\log \left(\gamma_{d} D_k \right)}{4 d}\biggr)^{r - \rho(k)} =\\
&&\frac{ \sqrt{2\pi}\,  2^{2r+\frac{1}{2}} e^{r}} {e^3  r^{r+1} }   \biggl(\frac{\log \log d}{2 \log d}\biggr)^{3 \rho(k)} \biggl(\frac{\log \left(\gamma_{d} D_k \right)}{4 d}\biggr)^{r - \rho(k)}>\\
&&\frac{ \sqrt{2\pi}\,  2^{2\rho(k)+\frac{1}{2}} e^{r}} {e^3  d^{2r+1-\rho(k)} }   \biggl(\frac{\log \log d}{2 \log d}\biggr)^{3 \rho(k)} \biggl(\log \left(\gamma_{d} D_k \right)\biggr)^{r - \rho(k)}.\end{eqnarray*}
Therefore, Theorem \ref{Sil-rho} gives a lower bound that  is larger than the lower bound  \eqref{friedmanremark} by at least a factor $e^{d/2} d^{-1/2} \left(\log \log d\right)^{3\rho(k)}$.


For the left-hand-side of \eqref{ineq-alpha}, we have
\begin{eqnarray*}
&&\frac{0.2}{r !} \, \biggl( \frac{2 d \, {\log \left(\gamma_{d}  D_k \right)}}{(d-2)\, d^{\log_{2}d}}\biggr)^{r - \rho(k)} 
>\\
& &0.2 e^{r-1}\, r^{-r-\frac{1}{2}} \,  \left(d^{\log_{2}d}\right)^{-r + \rho(k)} 
 \left(2 \, {\log \left(\gamma_{d}  D_k \right)}\right)^{r - \rho(k)} >\\
 & &0.2 \frac{e^{r-1}}{ \left(\frac{d^{\log_{2}d}}{2}\right)^{r - \rho(k)} } \, d^{-r-\frac{1}{2}} \,  
 \left( {\log \left(\gamma_{d}  D_k \right)}\right)^{r - \rho(k)}.
\end{eqnarray*}
Therefore, the lower bound obtained in Theorem \ref{Sil-alpha} is larger than the lower bound  \eqref{friedmanremark} by  at least a factor
$$\frac{e^{d/2} \left( \log d\right)^{3\rho(k)}}{ \left(\frac{d^{\log_{2}d}}{2d}\right)^{ r- \rho(k)}} .$$


In Theorems \ref{Sil-rho} and \ref{Sil-alpha}, we assume that 
$1 < \gamma_{d} D_{k}$.
Suppose that for a number field $k$ of degree $d$, we have $\gamma_d D_k   \leq 1$, where $\gamma_d$ is any of the values assumed in 
Theorems \ref{Sil-rho} and \ref{Sil-alpha}.  Then  by \eqref{intro61}, we have 
$$
\log D_k   < 5 \, \log \gamma_{d }^{-1} \, \Reg(k).
$$
This gives a stronger  lower bound for the regulators of number fields with small absolute discriminant than those stated  in our main theorems above.



This manuscript is organized as follows.
Section \ref{prelim} is a preliminary one and contains an overview of the Weil and Arakelov heights.  
In Section \ref{heightReg} we recall some lower bounds for the regulators and relative regulators  in terms of a product of  heights of ordinary and relative units. 
In Section \ref{heightdisc} for an algebraic 
number field $k$ of degree $d$, we obtain inequalities that relate Arakelov heights defined on $k^{d}$ and the absolute discriminant of $k$.
In Section \ref{specialheight} we prove inequalities relating  the Weil and Arakelov heights.
Section \ref{therhomethod} includes the proof of Theorem \ref{Sil-rho}, and Section 
\ref{thealphafunction} includes the proof of Theorem \ref{Sil-alpha}.

\section{The Weil and Arakelov heights}\label{prelim}

Let $k$ be an algebraic number field of degree $d$ over $\Q$.  At each place $v$ of $k$ we write $k_v$ for the completion of $k$ at $v$.
We work with two distinct absolute values $\|\ \|_v$ and $|\ |_v$ from each place $v$.  These are related by
\begin{equation*}\label{B4}
\|\ \|_v^{d_v/d} = |\ |_v,
\end{equation*}
where $d_v = [k_v : \Q_v]$ is the local degree at $v$, and $d = [k : \Q]$ is the global degree.  If $v | \infty$ then the restriction of $\|\ \|_v$ to 
$\Q$ is the usual archimedean absolute value on $\Q$, and if $v | p$ then the restriction of $\|\ \|_v$ to $\Q$ is the usual $p$-adic absolute value on $\Q$.
Then the absolute logarithmic Weil height is the map
\begin{equation*}\label{B6}
h : k^{\times} \rightarrow [0, \infty)
\end{equation*}
defined at each algebraic number $\alpha \not= 0$ in $k$ by the sum
\begin{equation}\label{B12}
h(\alpha) = \sum_v \log^+ |\alpha|_v = \hh \sum_v \bigl|\log |\alpha|_v \bigr|.
\end{equation}
In both sums there are only finitely many nonzero terms, and the equality on the right of (\ref{B12}) follows from the product formula.  It can be shown 
that the value of $h(\alpha)$ does not depend on the field $k$ that contains $\alpha$.  Hence the Weil height may be regarded as a map
\begin{equation*}\label{B14}
h : \oq \rightarrow [0, \infty).
\end{equation*}

Let $N \in \mathbb{N}$. At each place $v$ of $k$ we define a norm
\begin{equation*}\label{B50}
\|\ \|_v : k_v^{N+1} \rightarrow [0, \infty)
\end{equation*}
on (column) vectors $\bxi = (\xi_n)$ by
\begin{equation*}\label{B55}
\|\bxi\|_v = \begin{cases}    \bigl(\|\xi_0\|_v^2 + \|\xi_1\|_v^2 + \|\xi_2\|_v^2 + \cdots + \|\xi_N\|_v^2\bigr)^{\h}&    \text{if $v | \infty$,}\\
                                                           &   \\
                                                  \max\big\{\|\xi_0\|_v, \|\xi_1\|_v, \|\xi_2\|_v, \dots , \|\xi_N\|_v\big\}&      \text{if $v \nmid \infty$.}\end{cases}
\end{equation*}
We define a second norm 
\begin{equation*}\label{B60}
|\ |_v : k_v^{N+1} \rightarrow [0, \infty)
\end{equation*}
at each place $v$ by setting 
\begin{equation*}\label{B65}
|\bxi|_v = \|\bxi\|_v^{d_v/d}.
\end{equation*}
A vector $\bxi \not= \bo$ in $k^{N+1}$ has finitely many coordinates, and it follows that 
\begin{equation*}\label{B70}
|\bxi|_v =1
\end{equation*}
for all but finitely many places $v$ of $k$.  Then the Arakelov height
\begin{equation*}\label{B75}
H : k^{N+1} \setminus \{\bo\} \rightarrow [1, \infty)
\end{equation*}
is defined by
\begin{equation*}\label{B80}
H(\bxi) = \prod_v |\bxi|_v.
\end{equation*}
If $\bxi \not= \bo$, and $\xi_m \not= 0$ is a 
nonzero coordinate of $\bxi$, then using the product formula we get
\begin{equation*}\label{B82}
1 = \prod_v |\xi_m|_v \le \prod_v |\bxi|_v = H(\bxi).
\end{equation*}
Thus $H$ takes values in the interval $[1, \infty)$.  If $\eta \not= 0$ belongs to $k$, and $\bxi \not= \bo$ is a vector in $k^{N+1}$, then a second 
application of the product formula shows that
\begin{equation*}\label{B85}
H(\eta \bxi) = \prod_v |\eta \bxi|_v = \prod_v |\eta|_v |\bxi|_v = \prod_v |\bxi|_v = H(\bxi).
\end{equation*}
More information about the Arakelov height is contained in \cite{bombieri2006}.

\section{Weil  Heights and Regulators}\label{heightReg}

Throughout this section we suppose that  $k$ and $l$ are algebraic number fields with $k \subseteq l$.  We write $r(k)$ for the rank of the unit
group $O_k^{\times}$, and $r(l)$ for the rank of the unit group $O_l^{\times}$.  Then $k$ has $r(k) + 1$ 
archimedean places, and $l$ has $r(l) + 1$ archimedean places.  In general we have $r(k) \le r(l)$, and we recall 
(see \cite[Proposition 3.20]{narkiewicz2010}) that $r(k) = r(l)$ if and only if $l$ is a $\CM$-field, and $k$ is the maximal 
totally real subfield of $l$.  

The norm is a homomorphism of multiplicative groups
\begin{equation*}\label{unit1}
\Norm_{l/k} : l^{\times} \rightarrow k^{\times}.
\end{equation*}
If $v$ is a place of $k$, then each element $\alpha$ in $l^{\times}$ satisfies the identity
\begin{equation*}\label{unit3}
[l : k] \sum_{w|v} \log |\alpha|_w = \log |\Norm_{l/k}(\alpha)|_v.
\end{equation*}
It follows that the norm, restricted to the subgroup $O_l^{\times}$ of units, is a homomorphism
\begin{equation*}\label{unit5}
\Norm_{l/k} : O_l^{\times} \rightarrow O_k^{\times},
\end{equation*}
and the norm, restricted to the torsion subgroup in $O_l^{\times}$, is also a homomorphism
\begin{equation*}\label{unit7}
\Norm_{l/k} : \Tor\bigl(O_l^{\times}\bigr) \rightarrow \Tor\bigl(O_k^{\times}\bigr).
\end{equation*}
Therefore we get a well defined homomorphism, which we write as
\begin{equation*}\label{unit9}
\norm_{l/k} : O_l^{\times}/\Tor\bigl(O_l^{\times}\bigr) \rightarrow O_k^{\times}/\Tor\bigl(O_k^{\times}\bigr),
\end{equation*}
and define by
\begin{equation*}\label{unit11}
\norm_{l/k}\bigl(\alpha \Tor\bigl(O_l^{\times}\bigr)\bigr) = \Norm_{l/k}(\alpha) \Tor\bigl(O_k^{\times}\bigr).
\end{equation*}
However, to simplify notation we write
\begin{equation*}\label{unit15}
F_k = O_k^{\times}/\Tor\bigl(O_k^{\times}\bigr),\quad\text{and}\quad F_l = O_l^{\times}/\Tor\bigl(O_l^{\times}\bigr),
\end{equation*}
and we write the elements of the quotient groups $F_k$ and $F_l$ as coset representatives rather than cosets.
Obviously $F_k$ and $F_l$ are free abelian groups of rank $r(k)$ and $r(l)$, respectively. 

Following Costa and Friedman \cite{costa1991}, the subgroup of relative units in $O_l^{\times}$ is defined by
\begin{equation*}\label{unit19}
\big\{\alpha \in O_l^{\times} : \Norm_{l/k}(\alpha) \in \Tor\bigl(O_k^{\times}\bigr)\big\}.
\end{equation*}
Alternatively, we work in the free group $F_l$ where the image of the subgroup of relative units is the kernel of the homomorphism 
$\norm_{l/k}$.  That is, we define the subgroup of {\it relative units} in $F_l$ to be the subgroup
\begin{equation*}\label{unit21}
E_{l/k} = \big\{\alpha \in F_l : \norm_{l/k}(\alpha) = 1\big\}.
\end{equation*}
We also write
\begin{equation*}\label{unit23}
I_{l/k} = \big\{\textrm{norm}_{l/k}(\alpha) : \alpha \in F_l\big\} \subseteq F_k
\end{equation*}
for the image of the homomorphism $\norm_{l/k}$.  If $\beta$ in $F_l$ represents a coset in the subgroup $F_k$, then 
we have
\begin{equation*}\label{unit25}
\norm_{l/k}(\beta) = \beta^{[l : k]}.
\end{equation*}  
Therefore the image $I_{l/k} \subseteq F_k$ is a subgroup of rank $r(k)$, and the index satisfies
\begin{equation*}\label{unit27}
[F_k : I_{l/k}] < \infty.
\end{equation*}
It follows that $E_{l/k} \subseteq F_l$ is a subgroup of rank $r(l/k) = r(l) - r(k)$, and we restrict our attention here to extensions 
$l/k$ such that $r(l/k)$ is positive.    

Let $\eta_1, \eta_2, \dots , \eta_{r(l/k)}$ be a collection of multiplicatively independent relative units that form a basis for the subgroup 
$E_{l/k}$.  At each archimedean place $v$ of $k$ we select a place $\widehat{w}_v$ of $l$ such that $\widehat{w}_v | v$.  Then we 
define an $r(l/k) \times r(l/k)$ real matrix 
\begin{equation*}\label{unit32}
M_{l/k} = \bigl([l_w : \Q_w] \log \|\eta_j\|_w\bigr),
\end{equation*}
where $w$  is an archimedean place of $l$, but $w \not= \widehat{w}_v$ for each $v|\infty$, $w$ indexes rows, and 
$j = 1, 2, \dots , r(l/k)$ indexes columns.  We write $l_w$ for the completion of $l$ at the place $w$, $\Q_w$ for the
completion of $\Q$ at the place $w$, and we write $[l_w : \Q_w]$ for the local degree.  Of course $\Q_w$ is isomorphic
to $\R$ in the situation considered here.  As in \cite{costa1991}, we define the {\it relative regulator} of the extension $l/k$ to be 
the positive number
\begin{equation}\label{unit34}
\Reg\bigl(E_{l/k}\bigr) = \bigl|\det M_{l/k}\bigr|.
\end{equation}
It follows, as in the proof of \cite[Theorem 1]{costa1991} (see also \cite{costa1993}), that the
value of the determinant on the right of (\ref{unit34}) does not depend on the choice of places $\widehat{w}_v$ for 
each archimedean place $v$ of $k$.

It follows from \cite[Theorem 1.2]{akhtari2016} that there exist multiplicatively independent elements 
$\beta_1, \beta_2, \dots , \beta_{r(k)}$ in $F_k$ such that
\begin{equation}\label{short440}
\prod_{i = 1}^{r(k)} \bigl([k : \Q] h(\beta_i)\bigr) \le r(k)! \Reg(k).
\end{equation}

Suppose $k, l$ are distinct algebraic number fields, that $k$ is not $\Q$, $k$ is not an imaginary quadratic extension of $\Q$,
and $r(l) > r(k)$.  In \cite[Theorem 1.1]{akhtari2021} it is shown that that there exist multiplicatively independent elements $\psi_1, \psi_2, \dots , \psi_{r(l/k)}$
in the group $E_{l/k}$ of relative units such that
\begin{equation}\label{short445}
\prod_{j = 1}^{r(l/k)} \bigl([l : \Q] h(\psi_j)\bigr) \le r(l/k)! \Reg\bigl(E_{l/k}\bigr).
\end{equation}
It is shown in \cite{akhtari2021} that the two sets of multiplicatively independent units in \eqref{short440} and \eqref{short445} can be combined. The following is Corollary 1.2 of  \cite{akhtari2021}.
\begin{proposition}\label{lemshort2}  Let $\beta_1, \beta_2, \dots , \beta_{r(k)}$ be multiplicatively independent units in $F_k$
that satisfy {\rm (\ref{short440})}, and let $\psi_1, \psi_2, \dots , \psi_{r(l/k)}$ be multiplicatively independent units in $E(l/k)$ that satisfy
{\rm (\ref{short445})}.  Then the elements in the set
\begin{equation*}\label{short450}
\big\{\beta_1, \beta_2, \dots , \beta_{r(k)}\big\} \cup \big\{\psi_1, \psi_2, \dots , \psi_{r(l/k)}\big\}
\end{equation*}
are multiplicatively independent units in $F_l$, and they satisfy
\begin{equation*}\label{short455}
\prod_{i = 1}^{r(k)} \bigl([k : \Q] h(\beta_i)\bigr) \prod_{j = 1}^{r(l/k)} \bigl([l : \Q] h(\psi_j)\bigr) \le r(l)! \Reg(l).
\end{equation*}
\end{proposition}

\section{Arakelov heights and discriminants}\label{heightdisc}

In this section we suppose that $k  \subseteq \oQ$, where $\oQ$ is a fixed algebraic closure of $\Q$.  Then we write 
$\sigma_1, \sigma_2, \dots , \sigma_d$, for the distinct embeddings
\begin{equation*}\label{third5}
\sigma_j : k \rightarrow \oQ.
\end{equation*}
If $\Bbeta = (\beta_i)$ is a (column) vector in $k^d$ we define the $d \times d$ matrix
\begin{equation}\label{third8}
M(\Bbeta) = \bigl(\sigma_j(\beta_i)\bigr),
\end{equation}
where $i = 1, 2, \dots , d$, indexes rows and $j = 1, 2, \dots , d$, indexes columns.  We also define
\begin{equation*}\label{third11}
\B(k) = \big\{\Bbeta = (\beta_i) \in k^d : \text{$\beta_1, \beta_2, \dots , \beta_d$ are $\Q$-linearly independent}\big\}.
\end{equation*}
Then the matrix $M(\Bbeta)$ is nonsingular if and only if $\Bbeta$ belongs
to $\B(k)$.  Moreover, if $\alpha \not= 0$ belongs to $k$ then
\begin{equation}\label{third13}
\det M(\alpha \Bbeta) = \Norm_{k/\Q}(\alpha) \det M(\Bbeta),
\end{equation}
and if $A$ is a $d \times d$ matrix in the general linear group $\GL(d, \Q)$ we find that
\begin{equation}\label{third15}
\det M(A \Bbeta) = \det \bigl(A M(\Bbeta)\bigr) = \det A \det M(\Bbeta).
\end{equation}
These results are proved in \cite[Proposition 2.9]{narkiewicz2010}.

The product
\begin{equation*}\label{third19}
M(\Bbeta) M(\Bbeta)^T = \bigl(\Trace_{k/\Q}(\beta_i \beta_j)\bigr)
\end{equation*}
is a $d \times d$ matrix with entries in $\Q$.  Therefore, if $\Bbeta$ belongs to $\B(k)$ then
\begin{equation}\label{third23}
(\det M(\Bbeta))^2 = \det \bigl(M(\Bbeta) M(\Bbeta)^T\bigr) = \det \bigl(\Trace_{k/\Q}(\beta_i \beta_j)\bigr)
\end{equation}
is a nonzero rational number, and if $\Bbeta$ also has entries in $O_k$ then (\ref{third23}) is a nonzero integer.  It will be convenient to define the function
\begin{equation*}\label{third32}
f_k : \B(k) \rightarrow [0, \infty)
\end{equation*}
by
\begin{equation}\label{third37}
f_k(\Bbeta) = \big\|\det\bigl( M(\Bbeta) M(\Bbeta)^T\bigr)\big\|_{\infty} \prod_{v \nmid \infty} \|\Bbeta\|_v^{2 d_v}.
\end{equation}
Here $\|\ \|_{\infty}$ is the usual archimedean absolute value on $\Q$, and the product on the right of (\ref{third37}) is over the set of all
nonarchimedean places $v$ of $k$.  If $\alpha \not= 0$ belongs to $k$ and $\Bbeta$ belongs to $\B(k)$, then it follows using (\ref{third13})
and the product formula that
\begin{equation}\label{third43}
\begin{split}
f_k(\alpha \Bbeta) &= \big\|\det M(\alpha \Bbeta) M(\alpha \Bbeta)^T\big\|_{\infty} \prod_{v \nmid \infty} \|\alpha \Bbeta\|_v^{2 d_v}\\
    &= \biggl(\|\Norm_{k/\Q}(\alpha)\|_{\infty}^2 \prod_{v \nmid \infty} \|\alpha\|_v^{2 d_v}\biggr) 
    		\big\|\det M(\Bbeta) M(\Bbeta)^T\big\|_{\infty}^2 \prod_{v \nmid \infty} \|\Bbeta\|_v^{2d_v} \\
    &= \biggl(\prod_{v \mid \infty} \|\alpha\|_v^{2 d_v} \prod_{v \nmid \infty} \|\alpha\|_v^{2 d_v}\biggr) f_k(\Bbeta) = f_k(\Bbeta).
\end{split}
\end{equation}

For $\Bbeta = (\beta_i)$ in $\B(k)$, the fractional ideal generated by $\beta_1, \beta_2, \dots , \beta_d$, is the subset
\begin{equation}\label{third111}
\fJ(\Bbeta) = \big\{\eta \in k : \text{$\|\eta\|_v \le \|\Bbeta\|_v$ at each $v \nmid \infty$}\big\}.
\end{equation}
And the $\Z$-module generated by $\beta_1, \beta_2, \dots , \beta_d$ is
\begin{equation}\label{third113}
\M(\Bbeta) = \big\{\bxi^T \Bbeta = \xi_1 \beta_1 + \xi_2 \beta_2 + \cdots + \xi_d \beta_d : \bxi \in \Z^d\big\}.
\end{equation}
It is obvious that $\M(\Bbeta)$ is a subgroup of $\fJ(\Bbeta)$, and both $\M(\Bbeta)$ and $\fJ(\Bbeta)$ are free abelian groups of rank $d$.  
Hence the index $\bigl[\fJ(\Bbeta) : \M(\Bbeta)\bigr]$ is finite.  If $\alpha \not= 0$ belongs to $k$ and $\Bbeta$ is a vector in $\B(k)$ then 
using (\ref{third111}) we find that
\begin{equation}\label{third115}
\fJ(\alpha \Bbeta) = \big\{\eta \in k : \text{$\|\eta\|_v \le \|\alpha\|_v \|\Bbeta\|_v$ at each $v \nmid \infty$}\big\} = \alpha \fJ(\Bbeta),
\end{equation}
and in a similar manner we get
\begin{equation}\label{third118}
\M(\alpha \Bbeta) = \alpha \M(\Bbeta).
\end{equation}
Then it follows from (\ref{third115}) and (\ref{third118}) that 
\begin{equation}\label{third121}
\alpha \mapsto \bigl[\fJ(\alpha \Bbeta) : \M(\alpha \Bbeta)\bigr]
\end{equation}
is constant for $\alpha \not= 0$ in $k$.

Our next result shows that $f_k$ takes positive integer values on $\B(k)$ and provides a useful upper bound for the absolute discriminant.

\begin{proposition}\label{thmthird1}  Let $\Bbeta = (\beta_i)$ belong to $\B(k)$.   Let $\fJ(\Bbeta)$ be the fractional
ideal generated by $\beta_1, \beta_2, \dots , \beta_d$ as in {\rm (\ref{third111})}, and let $\M(\Bbeta)$ be the $\Z$-module generated by
$\beta_1, \beta_2, \dots , \beta_d$ as in {\rm (\ref{third113})}.  Then we have
\begin{equation}\label{third129}
f_k(\Bbeta) = \bigl[\fJ(\Bbeta) : \M(\Bbeta)\bigr]^2 D_k \le H(\Bbeta)^{2d},
\end{equation}
where $D_k$ is the absolute discriminant of $k$, and $\bigl[\fJ(\Bbeta) : \M(\Bbeta)\bigr]$ is the index of $\M(\Bbeta)$ in $\fJ(\Bbeta)$.
\end{proposition}

\begin{proof}  First we prove the equality on the left of (\ref{third129}).  And we assume to begin with that $\fJ(\Bbeta)$ is an integral ideal, or 
equivalently that
\begin{equation*}\label{third134}
\|\Bbeta\|_v \le 1\quad\text{at each nonarchimedean place $v$ of $k$}.
\end{equation*}
Let $\gamma_1, \gamma_2, \dots , \gamma_d$ be a basis for $\fJ(\Bbeta)$ as a $\Z$-module, and write
$\bgamma = (\gamma_j)$ for the corresponding vector in $\B(k)$.  By a basic identity for the discriminant of an integral ideal, 
see \cite[Proposition 2.13]{narkiewicz2010}, we have
\begin{equation}\label{third139}
\|\det M(\bgamma) M(\bgamma)^T\|_{\infty} = \bigl(\norm_{k/\Q} \fJ(\Bbeta)\bigr)^2 D_k = \bigl[O_k : \fJ(\Bbeta)\bigr]^2 D_k,
\end{equation}
where 
\begin{equation*}\label{third141}
M(\bgamma) = \bigl(\sigma_j(\gamma_i)\bigr) 
\end{equation*}
is the $d \times d$ matrix defined as in \eqref{third8}.  As $\beta_1, \beta_2, \dots , \beta_d$ belong to $\fJ(\Bbeta)$ there exists a unique, nonsingular,
$d \times d$ matrix $A = (a_{i j})$ with entries in $\Z$ such that
\begin{equation}\label{third143}
\beta_i = \sum_{j = 1}^d a_{i j} \gamma_j,\quad\text{or equivalently $\Bbeta = A \bgamma$}.
\end{equation}
It follows from (\ref{third111}) that $\|\bgamma\|_v \le \|\Bbeta\|_v$ for each $v \nmid \infty$, and it follows from (\ref{third143}) and the strong triangle
inequality that $\|\Bbeta\|_v \le \|\bgamma\|_v$ for each $v \nmid \infty$.  Then from (\ref{third143}) we also get
\begin{equation}\label{third145}
\bigl[\fJ(\Bbeta) : \M(\Bbeta)\bigr] = \|\det A\|_{\infty}.
\end{equation}
As $\fJ(\Bbeta)$ is an integral ideal generated (as an ideal) by $\beta_1, \beta_2, \dots , \beta_d$ and also generated (as a $\Z$-module)
by $\gamma_1, \gamma_2, \dots , \gamma_d$,  we have
\begin{equation}\label{third147}
\prod_{v \nmid \infty} \|\Bbeta\|_v^{-d_v} = \prod_{v \nmid \infty} \|\bgamma\|_v^{-d_v} = \norm_{k/\Q} \fJ(\Bbeta) = \bigl[O_k : \fJ(\Bbeta)\bigr].
\end{equation}
We combine (\ref{third15}), (\ref{third139}), (\ref{third145}) and (\ref{third147}), and conclude that
\begin{equation*}\label{third151}
\begin{split}
\big\|\det M(\Bbeta) M(\Bbeta)^T\big\|_{\infty}&\prod_{v \nmid \infty} \|\Bbeta\|_v^{2d_v}\\
                    &= \big\|\det M(A \bgamma) M(A \bgamma)^T\big\|_{\infty} \prod_{v \nmid \infty} \|\bgamma\|_v^{2d_v}\\
                    &= \|\det A\|_{\infty}^2 \big\|\det M(\bgamma) M(\bgamma)^T\big\|_{\infty} \bigl[O_k : \fJ(\Bbeta)\bigr]^{-2}\\
                    &= \bigl[\fJ(\Bbeta) : \M(\Bbeta)\bigr]^2 D_k.
\end{split}
\end{equation*}
This proves the equality on the left of (\ref{third129}) under the assumption that $\fJ(\Bbeta)$ is an integral ideal.
 
If $\fJ(\Bbeta)$ is a fractional ideal in $k$, but not necessarily an integral ideal, then there exists an algebraic integer $\alpha \not= 0$ in $O_k$ such that 
$\alpha \fJ(\Bbeta) = \fJ( \alpha \Bbeta)$ is an integral ideal.  Therefore we get the identity
\begin{equation}\label{third153}
f_k(\alpha \Bbeta) = \bigl[\fJ(\alpha \Bbeta) : \M(\alpha \Bbeta)\bigr]^2 D_k
\end{equation}
by the case already considered.  We use (\ref{third43}), (\ref{third121}), and (\ref{third153}), to establish the equality on the left of 
(\ref{third129}) in general.

Next we prove the inequality on the right of (\ref{third129}).  We assume that $\oQ \subseteq \C$, and write $|\ |$ for the usual Hermitian absolute value
on $\C$.  Each embedding
\begin{equation*}\label{third166}
\sigma_j : k \rightarrow \oQ \subseteq \C
\end{equation*}
determines an archimedean place $v$ of $k$ such that
\begin{equation*}\label{third171}
\|\eta\|_v = |\sigma_j(\eta)|\quad\text{for $\eta$ in $k$.}
\end{equation*}
As $j = 1, 2, \dots , d$, each real archimedean place $v$ occurs once and each complex archimedean place $v$ occurs twice.
Then Hadamard's inequality applied to the matrix $M(\Bbeta) = \bigl(\sigma_j(\beta_i)\bigr)$ leads to
\begin{equation}\label{third176}
\begin{split}
\|\det M(\Bbeta) M(\Bbeta)^T\|_{\infty} &= \bigl|\det \bigl(\sigma_j(\beta_i)\bigr)\bigr|^2\\
                        &\le \prod_{j = 1}^d \biggl(\sum_{i = 1}^d |\sigma_j(\beta_i)|^2\biggr)\\
                        &= \prod_{v | \infty} \biggl(\sum_{i = 1}^d \|\beta_i\|_v^2\biggr)^{d_v}\\
                        &= \prod_{v | \infty} \|\Bbeta\|_v^{2 d_v}.
\end{split}
\end{equation}
It follows from (\ref{third176}) that
\begin{equation}\label{third181}
\begin{split}
f_k(\Bbeta) &= \|\det M(\Bbeta) M(\Bbeta)^T\|_{\infty} \prod_{v \nmid \infty} \|\Bbeta\|_v^{2 d_v}\\
	&\le \prod_{v | \infty} \|\Bbeta\|_v^{2 d_v} \prod_{v \nmid \infty} \|\Bbeta\|_v^{2 d_v} = H(\Bbeta)^{2 d}.
\end{split}
\end{equation}
Now (\ref{third181}) verifies the inequality on the right of (\ref{third129}). 
\end{proof}

\section{Special height inequalities}\label{specialheight}

In this section we present inequalities where the Arakelov height $H(\balpha)$ is bounded by the Weil height of the coordinates of $\balpha$.
Such inequalities are useful when $H$ is applied to vectors having coordinates that satisfy simple algebraic conditions.

\begin{lemma}\label{lemB1}  Let $k$ be an algebraic number field and let $\alpha \not= 0$ be a point in $\oQ$ such that 
$M = [k(\alpha) : k]$.  Let $\ba = \bigl(\alpha^{m-1}\bigr)$ be the column vector in $k^M$ where $m = 1, 2, \dots , M$, indexes rows.  
Then we have
\begin{equation}\label{B107}
\log H(\ba) \le \hh \log M + (M-1)h(\alpha).
\end{equation}
\end{lemma}

\begin{proof}  Let $l$ be an algebraic number field such that $k \subseteq k(\alpha) \subseteq l$ and let $w$ be a place of $l$.  If 
$w \nmid \infty$ we find that
\begin{equation}\label{B110}
|\ba|_w = \max\big\{1, |\alpha|_w, \dots , |\alpha|_w^{M - 1}\big\} = \max\big\{1, |\alpha|_w\big\}^{(M - 1)}.
\end{equation}
If $w | \infty$ we get
\begin{equation*}\label{B115}
\begin{split}
\|\ba\|_w = \Bigl(1 + \|\alpha\|_w^2 + \|\alpha\|_w^4 + \cdots + \|\alpha\|_w^{2M - 2}\Bigr)^{\h}
                   \le M^{\h} \max\big\{1, \|\alpha\|_w\big\}^{(M - 1)},
\end{split}
\end{equation*}                   
and then                   
\begin{equation}\label{B120}
\log |\ba|_w \le \frac{[l_w : \Q] \log M}{2 [l : \Q]} + (M - 1) \log ^+ |\alpha|_w.
\end{equation}
Combining (\ref{B110}) and (\ref{B120}), we find that
\begin{equation*}\label{B125}
\begin{split}
\log H(\ba) &= \sum_w \log |\ba|_w\\ 
		&\le \sum_{w | \infty} \frac{[l_w : \Q_w] \log M}{2 [l : \Q]} + (M - 1) \sum_w \log^+ |\alpha|_w\\
		&= \hh \log M + (M - 1) h(\alpha).
\end{split}
\end{equation*}
This verifies the inequality (\ref{B107}).
\end{proof}

If $K$ is a field and $K(\alpha)$ is a simple, algebraic extension of $K$ of positive degree $N$, then every element $\eta$ in $K(\alpha)$ has 
a unique representation of the form
\begin{equation*}\label{B17}
\eta = \sum_{n=0}^{N-1} c(n)\alpha^n,\quad\text{where}\quad c(n)\in K.
\end{equation*}
This extends to fields obtained by adjoining finitely many algebraic elements using a simple inductive argument.

\begin{lemma}\label{lemB2}  Let $K \subseteq L$ be fields, let $\alpha_1, \alpha_2, \dots , \alpha_M$, be elements of $L$,
and assume that each $\alpha_m$ is algebraic over $K$.  Define positive integers $N_m$ by
\begin{equation}\label{B19}
N_1 = [K(\alpha_1) : K],
\end{equation}
and by
\begin{equation}\label{B20}
N_m = [K(\alpha_1, \alpha_2, \dots , \alpha_m):K(\alpha_1, \alpha_2, \dots , \alpha_{m-1})]
\end{equation}
for $m = 2, 3, \dots , M$.
Then every element $\eta$ in $K(\alpha_1, \alpha_2, \dots , \alpha_M)$ has a unique representation of the form 
\begin{equation}\label{B21}
\eta = \sum_{n_1 = 0}^{N_1-1}\sum_{n_2 = 0}^{N_2-1} \cdots \sum_{n_M = 0}^{N_M-1} 
	c(\bn) \alpha_1^{n_1}\alpha_2^{n_2} \cdots \alpha_M^{n_M},\quad\text{where}\quad c(\bn)\in K.
\end{equation}
Moreover, $K(\alpha_1, \alpha_2, \dots , \alpha_M)/K$ is a finite extension of degree $N_1 N_2 \cdots N_M$, and the elements in the set
\begin{equation}\label{B21.2}
\big\{\alpha_1^{n_1}\alpha_2^{n_2} \cdots \alpha_M^{n_M}: 0 \le n_m < N_m,\ m=1, 2, \dots , M\big\}
\end{equation}
form a basis for $K(\alpha_1, \alpha_2, \dots , \alpha_M)$ as a vector space over $K$.
\end{lemma}

\begin{proof}  We argue by induction on $M$.  If $M = 1$ then the result is well known.  Therefore we assume that $M \ge 2$.  As 
\begin{equation*}\label{B21.7}
K(\alpha_1, \dots , \alpha_{M-1}, \alpha_M)/K(\alpha_1, \dots , \alpha_{M-1})
\end{equation*}
is a simple extension, the element $\eta$ in $K(\alpha_1, \dots , \alpha_{M-1}, \alpha_M)$ has a unique representation of the form
\begin{equation}\label{B22}
\eta = \sum_{n_M = 0}^{N_M-1} a(n_M) \alpha_M^{n_M},
		\quad\text{where}\quad a(n_M)\in K(\alpha_1, \alpha_2, \dots , \alpha_{M-1}).
\end{equation}
By the inductive hypothesis each coefficient $a(n_M)$ has a representation in the form
\begin{equation}\label{B23}
a(n_M) = \sum_{n_1 = 0}^{N_1-1}\sum_{n_2 = 0}^{N_2-1} \cdots \sum_{n_{M-1} = 0}^{N_{M-1}-1} 
	b(\bn^{\prime}, n_M) \alpha_1^{n_1}\alpha_2^{n_2} \cdots \alpha_{M-1}^{n_{M-1}},
\end{equation}
where each $b(\bn^{\prime}, n_M)$ belongs to $K$.  When the sum on the right of (\ref{B23}) is inserted into (\ref{B22}), we 
obtain the representation (\ref{B21}). 

We have proved that the set (\ref{B21.2}) spans the field $K(\alpha_1, \alpha_2, \dots , \alpha_M)$ as a vector space over $K$.  Clearly the 
set (\ref{B21.2}) has cardinality at most $N_1N_2 \cdots N_M$.  Because
\begin{equation*}\label{B25}
K \subseteq K(\alpha_1) \subseteq K(\alpha_1, \alpha_2) \subseteq \cdots \subseteq K(\alpha_1, \alpha_2, \cdots , \alpha_M),
\end{equation*}
it follows from (\ref{B19}) and (\ref{B20}) that
\begin{equation*}\label{B27}
[K(\alpha_1, \alpha_2, \dots , \alpha_M):K] = N_1N_2 \cdots N_M.
\end{equation*}
We conclude that the set (\ref{B21.2}) is a basis for $K(\alpha_1, \alpha_2, \cdots , \alpha_M)$ over $K$.  Therefore the representation 
(\ref{B21}) is unique.
\end{proof}

Let $k$ and $l$ be distinct algebraic number fields such that $k \subseteq l$.  We establish a bound for $H(\Bbeta)$ in the special case 
where the coordinates of $\Bbeta$ generate the field extension $l/k$.  We assume that $\alpha_1, \alpha_2, \dots , \alpha_M$, 
are algebraic numbers such that
\begin{equation*}\label{early500}
l = k(\alpha_1, \alpha_2, \dots , \alpha_M).
\end{equation*}
Then it follows from Lemma \ref{lemB2} that there exist positive integers $N_1, N_2, \dots , N_M$, such that
\begin{equation*}\label{early502}
N_1 N_2 \cdots N_M = [l : k],
\end{equation*}
and the elements of the set 
\begin{equation}\label{early504}
\big\{\alpha_1^{n_1}\alpha_2^{n_2} \cdots \alpha_M^{n_M}: 0 \le n_m < N_m,\ m=1, 2, \dots , M\big\}
\end{equation}
form a basis for $l = k(\alpha_1, \alpha_2, \dots , \alpha_M)$ as a vector space over $k$.  We define a tower of intermediate fields
\begin{equation}\label{early506}
k = k_0 \subseteq k_1 \subseteq k_2 \subseteq \cdots \subseteq k_M = l,
\end{equation}
by 
\begin{equation*}\label{early508}
k_m = k(\alpha_1, \alpha_2, \dots , \alpha_m),\quad\text{where $m = 1, 2, \dots , M$}.
\end{equation*}
Then it follows from (\ref{B20}) that
\begin{equation*}\label{early510}
N_m = [k_m : k_{m-1}] = [k_{m-1}(\alpha_m) : k_{m-1}],\quad\text{for each $m = 1, 2, \dots , M$},
\end{equation*}
and
\begin{equation*}\label{early512}
N_1 N_2 \cdots N_m = [k_m : k_0],\quad\text{for each $m = 1, 2, \dots , M$}.
\end{equation*}
We note that the tower of intermediate fields (\ref{early506}) depends on the ordering of the generators $\alpha_1, \alpha_2, \dots , \alpha_M$, 
and a permutation of these generators would (in general) change the intermediate fields in the tower. 

\begin{lemma}\label{lemB3}  Let $\Bbeta$ be the vector in $l^{[l : k]}$ such that the elements of the set {\rm (\ref{early504})} are the 
coordinates of $\Bbeta$.  For each $m = 1, 2, \dots , M$, let $\ba_m$ be the vector in $k_m^{N_m}$ defined by
$\ba_m = \bigl(\alpha_m^{n_m}\bigr)$ where $n_m = 0, 1, \dots , N_m - 1$.
Then we have
\begin{equation}\label{early525}
H(\Bbeta) = \prod_{m = 1}^M H(\ba_m),
\end{equation}
and
\begin{equation}\label{early527}
\log H(\Bbeta) \le \hh \log [l : k] + \sum_{m = 1}^M (N_m - 1) h(\alpha_m).
\end{equation}
\end{lemma}

\begin{proof}  At each archimedean place $w$ of $l$ we have
\begin{equation}\label{early530}
\begin{split}
\prod_{w|\infty} \|\Bbeta\|_w &= \prod_{w|\infty} \biggl(\sum_{n_1 = 0}^{N_1-1} \sum_{n_2 = 0}^{N_2-1} \cdots \sum_{n_M = 0}^{N_M-1} 
	\|\alpha_1^{n_1}\|_w^2 \|\alpha_2^{n_2}\|_w^2 \cdots \|\alpha_M^{n_M}\|_w^2\biggr)^{\h}\\
	&= \prod_{w | \infty} \biggl(\prod_{m = 1}^M \sum_{n_m = 0}^{N_m-1} \|\alpha_m^{n_m}\|_w^2\biggr)^{\h}\\
	&= \prod_{m = 1}^M \prod_{w|\infty} \biggl(\sum_{n_m = 0}^{N_m-1} \|\alpha_m^{n_m}\|_w^2\biggr)^{\h}\\
	&= \prod_{m = 1}^M \prod_{w|\infty} \|\ba_m\|_w.
\end{split}
\end{equation}
At each nonarchimedean place $w$ of $l$ we find that
\begin{equation}\label{early535}
\begin{split}
\prod_{w\nmid \infty} \|\Bbeta\|_w &= \prod_{w|\infty} \max\big\{\|\alpha_1^{n_1} \alpha_2^{n_2} \cdots \alpha_M^{n_M}\|_w : 0 \le n_m < N_m\big\}\\
	&= \prod_{w \nmid \infty} \prod_{m = 1}^M \max\big\{\|\alpha_m^{n_m}\|_w : 0 \le n_m < N_m\big\}\\
	&= \prod_{m = 1}^M \prod_{w \nmid \infty} \max\big\{\|\alpha_m^{n_m}\|_w : 0 \le n_m < N_m\big\}\\
	&= \prod_{m = 1}^M \prod_{w\nmid \infty} \|\ba_m\|_w.
\end{split}
\end{equation}
Clearly (\ref{early525}) follows from (\ref{early530}) and (\ref{early535}).  Then using (\ref{early525}) and the inequality (\ref{B107}) we get
\begin{equation}\label{early545}
\begin{split}
\log H(\Bbeta) &= \sum_{m = 1}^M \log H(\ba_m)\\
	&\le \sum_{m = 1}^M \Bigl(\hh \log N_m + (N_m - 1)h(\alpha_m)\Bigr)\\
	&= \hh \log [l : k] + \sum_{m = 1}^M (N_m - 1) h(\alpha_m).
\end{split}
\end{equation}
This verifies (\ref{early527}).
\end{proof}

We conclude this section with an inequality that will be very useful in our proofs. Let $\alpha$ be an algebraic number, $m = [\mathbb{Q}(\alpha):\mathbb{Q}]$, 
and $D_{\mathbb{Q}(\alpha)}$ the absolute discriminant of the number field $\Q(\alpha)$.  From \eqref{third129} and \eqref{early527}, we obtain 
\begin{equation}\label{Dhm}
h(\alpha) \geq  \frac{\log \frac{D_{\mathbb{Q}(\alpha)}}{m^m}}{2m (m-1)}.
\end{equation}
A similar inequality has been established in \cite{silverman1983} by a different method. 

\section{A special intermediate field with large rank; \\ Proof of Theorem \ref{Sil-rho}}\label{therhomethod}

Suppose $k$ is a number field of degree $d$.  Let $r$ be the rank of the unit group $O_k^{\times}$ in $k$.  
By  \cite[Theorem 1.2]{akhtari2016} there exist multiplicatively independent
elements $\alpha_1, \alpha_2, \dots , \alpha_r$ in $O_k^{\times}$ such that
\begin{equation}\label{early361}
d^r \, \prod_{j = 1}^r h(\alpha_j) \le \frac{2^r (r!)^3}{(2r)!} \Reg(k),
\end{equation}
where $\Reg(k)$ is the regulator of $k$.  If we assume now that $k$ is not a $\CM$-field, then the rank of the unit group $O_k^{\times}$ is
strictly larger than the rank of the unit group in each proper subfield of $k$.  As the multiplicative group generated by 
$\alpha_1, \alpha_2, \dots , \alpha_r$ has rank equal to the rank of $O_k^{\times}$, it follows that
\begin{equation}\label{early365}
k = \Q(\alpha_1, \alpha_2, \dots , \alpha_r).
\end{equation} 
Applying Lemma \ref{lemB2} to (\ref{early365}), we conclude that there exist positive integers 
\begin{equation*}\label{early367}
N_1, N_2, \dots , N_r, 
\end{equation*}
and a corresponding tower of intermediate fields
\begin{equation*}\label{early370}
k_0 = \Q \subseteq k_1 \subseteq k_2 \subseteq k_3 \subseteq \cdots \subseteq k_r = k,
\end{equation*}
such that
\begin{equation}\label{early375}
k_j = \Q(\alpha_1, \alpha_2, \dots , \alpha_j),\quad\text{where $j = 1, 2, \dots , r$},
\end{equation}
\
\begin{equation*}\label{early380}
N_j = [k_j : k_{j-1}] = [k_{j-1}(\alpha_j) : k_{j-1}],\quad\text{for each $j = 1, 2, \dots , r$},
\end{equation*}
and
\begin{equation}\label{early385}
N_1 N_2 \cdots N_j = [k_j : \Q],\quad\text{for each $j = 1, 2, \dots , r$}.
\end{equation}
In particular, (\ref{early385}) with $j = r$ is also
\begin{equation*}\label{early388}
N_1 N_2 \cdots N_r = d.
\end{equation*}
Moreover, from (\ref{third129}) and (\ref{early545}) we get the inequality
\begin{equation}\label{early390}
\log D_k \le 2d \log H(\Bbeta) \le d \log d + 2d \sum_{j = 1}^r (N_j - 1) h(\alpha_j),
\end{equation}
where $\Bbeta$ is the vector in $\B(k)$ such that the elements of the set
\begin{equation*}\label{early400}
\big\{\alpha_1^{n_1} \alpha_2^{n_2} \cdots \alpha_r^{n_r} : \text{$0 \le n_j < N_j$, and $j = 1, 2, \dots , r$}\big\}
\end{equation*}
are the coordinates of $\Bbeta$.

It follows from (\ref{early375}) that the unit group $O_{k_j}^{\times}$ contains the collection of $j$ multiplicatively 
independent units $\alpha_1, \alpha_2, \dots , \alpha_j$.  Therefore we have
\begin{equation}\label{extra10} 
j \le \rank O_{k_j}^{\times},\quad\text{for each $j = 1, 2, \dots , r$}.
\end{equation}
As defined in \eqref{extra15}, let
\begin{equation*}\label{extra13}
\rho(k) = \max\big\{\rank O_{k^{\prime}}^{\times} : \text{$k^{\prime} \subseteq k$ and $k^{\prime} \not= k$}\big\}.
\end{equation*}
Because $k$ is not a $\CM$-field, we have $\rho(k) < r$.  It will also be convenient to define 
\begin{equation}\label{extra18}
q = \min\{j : \text{$1 \le j \le r$ and $k_j = k$}\} \le \rho(k) + 1,
\end{equation}
where the inequality on the right of (\ref{extra18}) follows from (\ref{extra10}) and the definition of $\rho(k)$.  Using the positive integer $q$
we find that
\begin{equation*}\label{extra27}
j \le \rank O_{k_j}^{\times} \le \rho(k),\quad\text{if and only if $1 \le j \le q - 1$},
\end{equation*}
and
\begin{equation*}\label{extra29}
k_j = k,\quad\text{if and only if $q \le j \le r$}.
\end{equation*}
It follows that
\begin{equation*}\label{extra32}
2 \le N_q = [k_q : k_{q-1}] = [k : k_{q-1}],
\end{equation*}
and
\begin{equation*}\label{extra33}
N_j = 1\quad\text{for $q < j \le r$}.
\end{equation*}
Thus the inequality (\ref{early390}) can be written as
\begin{equation*}\label{extra35}
\frac{\log D_k - d \log d}{2d} \le \sum_{j = 1}^q (N_j - 1) h(\alpha_j).
\end{equation*}
It is clear that an advantageous ordering of the independent units $\alpha_1, \alpha_2, \dots , \alpha_r$ would be
\begin{equation}\label{extra40}
0 < h(\alpha_1) \le h(\alpha_2) \le \dots \le h(\alpha_r),
\end{equation}
which we assume from now on.
Finally, as $N_1, N_2, \dots , N_q$ are positive integers, the inequality 
\begin{equation*}\label{extra42}
\sum_{j = 1}^q (N_j - 1) \le (N_1 N_2 \cdots N_q) - 1 = d - 1
\end{equation*}
is easy to verify by induction on $q$.  Then from \eqref{extra40} we get
\begin{equation*}\label{extra45}
\begin{split}
\frac{\log D_k - d \log d}{2d} &\le \sum_{j = 1}^q (N_j - 1) h(\alpha_j)\\
	               &\le h(\alpha_q) \sum_{j = 1}^q (N_j - 1)\\
		       &\le (d - 1) h(\alpha_q),
\end{split}
\end{equation*}
which we write as
\begin{equation}\label{extra51}
\frac{\log D_k - d \log d}{2d} \le d h(\alpha_q).
\end{equation}
Plainly the inequality (\ref{extra51}) is of interest if and only if
\begin{equation*}\label{extra56}
0 < \log D_k - d \log d,
\end{equation*}
which is also the hypothesis of Theorem \ref{Sil-rho}.  Then it follows from (\ref{extra40}) that
\begin{equation}\label{extra59}
\frac{\log D_k - d \log d}{2d} \le d h(\alpha_j).
\end{equation}
for each 
\begin{equation*}\label{extra60}
j = q, q+1, q+2, \dots , r.  
\end{equation*}
Since the value of $q$ is unknown and depends on the ordering (\ref{extra40}), we use (\ref{extra59})
in the more restricted range 
\begin{equation*}\label{extra63}
j = \rho(k) + 1, \rho(k) + 2, \dots , r.  
\end{equation*}
Then (\ref{extra40}) and (\ref{extra59}) imply that
\begin{equation}\label{extra67}
\biggl(\frac{\log D_k - d \log d}{2d}\biggr)^{r - \rho(k)} \le \prod_{j = \rho(k) + 1}^r \bigl(d h(\alpha_j)\bigr).
\end{equation}

In order to obtain the desired explicit bounds in Theorem \ref{Sil-rho}, we apply results of  Dobrowolski in \cite{dobrowolski1979} and Voutier in 
\cite{voutier1996}.  From \cite{dobrowolski1979} there exists a positive constant $c'(d)$, which
depends only on the degree $d = [k : \Q]$, such that the inequality
\begin{equation}\label{extra71}
c'(d) \le d h(\gamma)
\end{equation}
holds for algebraic numbers $\gamma$ in $k^{\times}$ which are not roots of unity.  Then from (\ref{early361}), (\ref{extra67}), and (\ref{extra71}),
we get
\begin{equation}\label{extra76}
\begin{split}
c'(d)^{\rho(k)} \biggl(\frac{\log D_k  - d \log d}{2d}\biggr)^{r - \rho(k)} 
		&\le \prod_{j = 1}^r \bigl(d h(\alpha_j)\bigr)\\
		&\le \frac{2^r (r!)^3}{(2r)!} \Reg(k).
\end{split}
\end{equation}
From \cite{voutier1996} we have
\begin{equation}\label{extra84}
\tfrac{1}{4}\biggl(\frac{\log \log d}{\log d}\biggr)^3 \le c'(d)
\end{equation}
for each number field $k \not= \Q$.  Hence (\ref{extra76}) and (\ref{extra84}) lead to the explicit inequality
\begin{equation*}\label{extra90}
\biggl(\frac{\log \log d}{2 \log d}\biggr)^{3 \rho(k)} \biggl(\frac{\log D_k  - d \log d}{4 d}\biggr)^{r - \rho(k)} 
		\le  \frac{(r!)^3}{(2r)!} \Reg(k).
\end{equation*} 
This completes the proof of Theorem \ref{Sil-rho}.
\bigskip

\textbf{Remark}. From the work of Amoroso and David \cite{AmDa} (see also Theorem  4.4.7 in \cite{bombieri2006}), we get
\begin{equation}\label{AmoDavid}
c\,  n^{-1}  (1 + \log n)^{- \rho \kappa} \le \prod_{i=1}^{\rho(k)} h(\alpha_{i}),
\end{equation}
where 
$$
n= [ \mathbb{Q}(\alpha_{1}, \ldots, \alpha_{\rho}): \mathbb{Q}] ,
$$
and $c$ and $\kappa$ depend only on the number of algebraic numbers in the product on the right hand side of \eqref{AmoDavid}, which 
in our case is $\rho = \rho(k)$. 
 
From (\ref{early361}), (\ref{extra67}),  (\ref{AmoDavid}), and since $n < d = [k: \Q]$,
we get
\begin{equation*}\label{extra7600}
\begin{split}
c\, d^{\rho(k)-1} (1 + \log d)^{- \rho(k) \kappa} \biggl(\frac{\log D_k  - d \log d}{2d}\biggr)^{r - \rho(k)} 
		&\le \prod_{j = 1}^r \bigl(d h(\alpha_j)\bigr)\\
		&\le \frac{2^r (r!)^3}{(2r)!} \Reg(k).
\end{split}
\end{equation*}
This implies the inequality \eqref{amorosodavidapplication}. A completely explicit version of  (\ref{AmoDavid}) is given in Corollary 1.6 of \cite{AmoV} and implies 
\begin{equation*}\label{amorosovpplication*}
\frac{(2r)!}{(r!)^3} \, \frac{d^{\rho(k)-1}} {\left(1050\, \rho(k)^5 \log (1.5)d\right)^{\rho^2(k) (\rho(k)+1)^2}} \biggl(\frac{\log \left(d^{-d} D_k \right)}{4d}\biggr)^{r - \rho(k)}
\le  \Reg(k),
\end{equation*}
where the dependence on $\rho(k)$ is unlikely to be optimal.

\section{A special intermediate field with optimal discriminant;\\ Proof of Theorem \ref{Sil-alpha}}\label{thealphafunction}

Let $k$ be an algebraic number field, and $\I( k)$  the set of
intermediate number fields $k'$ such that $\mathbb{Q} \subseteq k'  \subseteq k$.  We will define two maps 
$$
\lambda :  \I( k) \rightarrow \mathbb{N} \cup \{0\}
$$
and 
$$
\aleph :  \I( k) \rightarrow (0, 1].
$$
For each number field $k' \in \I(k)$ we define $\lambda\bigl(k'\bigr)$ to be the maximum length of a tower of subfields of $k$ that begins at 
$\mathbb{Q}$ and ends at $k'$, with $\lambda(\mathbb{Q}) = 0$.  If $k_1$ and $k_2$ are distinct intermediate fields such that 
\begin{equation*}\label{A1}
\mathbb{Q} \subseteq k_1 \subseteq k_2 \subseteq k,
\end{equation*}
then it is obvious that
\begin{equation}\label{A3}
\lambda(k_1) < \lambda(k_2) \leq \lambda(k) \leq \log_{2} d,
\end{equation}
where $d = [k:\mathbb{Q}]$.

For each number field $k'$ we write $D_{k'}$ for the absolute discriminant of $k'$.  For $k' \subseteq k$, we have 
(see \cite[Corollary to Proposition 4.15]{narkiewicz2010}) $D_{k'}^{[k : k']} \mid D_{k}$, and if  $k' \neq k$ we have
\begin{equation}\label{A10}
D_{k'} < D_{k}^{[k:k']^{-1}}.
\end{equation}
In order to better control the change in the absolute values of discriminants of intermediate fields, we normalize the exponent $[k:k']^{-1}$ in the 
above inequality.  For each subfield $k'$ of $k$, we define
\begin{equation}\label{A23}
\aleph \bigl(k'\bigr): = \bigl(2 \bigl[k : k'\bigr]\bigr)^{\lambda(k') - \lambda(k)}.
\end{equation}

First we prove two useful lemmas about properties of the function $\aleph$.
\begin{lemma}
Let
\begin{equation*}\label{A5}
\mathbb{Q} \subseteq k_1 \subseteq k_2 \subseteq \cdots \subseteq k_{N-1} \subseteq k_N = k
\end{equation*}
be a tower of length $N$, containing $N+1$ distinct number fields.
We have
\begin{equation*}\label{A34}
0 < \aleph \bigl(\mathbb{Q}\bigr) < \aleph \bigl(k_1\bigr) < \aleph \bigl(k_2\bigr) < \cdots < \aleph \bigl(k_{N-1}\bigr) < \aleph \bigl(k_N\bigr) = 1.
\end{equation*}
\end{lemma}
\begin{proof}
By the definition of the function $\aleph$ in \eqref{A23}, we have $\aleph(k) = 1$ and $\aleph \bigl(\mathbb{Q}\bigr)> 0$.  Now suppose that  
$k_1$, $k_2$ are distinct intermediate fields, with $\mathbb{Q} \subseteq k_{1} \subseteq  k_{2} \subseteq k$, by \eqref{A3} we have 
\begin{equation*}\label{A27}
\begin{split}
\aleph\bigl(k_1\bigr) &= \bigl(2 \, \bigl[k : k_1\bigr]\bigr)^{\lambda(k_1) - \lambda(k)}\\
	&= \bigl(2\, \bigl[k : k_2\bigr]\bigr)^{\lambda(k_2) - \lambda(k)} 			
	\bigl(2\,  \bigl[k : k_2\bigr]\bigr)^{\lambda(k_1) - \lambda(k_2)} \bigl(2\, \bigl[k_2 : k_1\bigr]\bigr)^{\lambda(k_1) - \lambda(k)}\\
	&< \bigl(2\,  \bigl[k : k_2\bigr]\bigr)^{\lambda(k_2) - \lambda(k)}\\
	&= \aleph \bigl(k_2\bigr).
\end{split}
\end{equation*}
\end{proof}

\begin{lemma}\label{why2}
Let $\Q \subseteq k' \subseteq k$. Assume $\alpha \in k$ and $\alpha \not \in k'$ so that $k' \subseteq k'(\alpha) \subseteq k$. We have
\begin{equation}\label{alphadif}
   \aleph(k'(\alpha)) - \aleph(k') [k: k']   \geq   2^{\lambda(k')- \lambda(k)}\left( [k: k']\right)^{\lambda(k')- \lambda(k) +1}.
    \end{equation}
\end{lemma}

\begin{proof}
Since  $\lambda(k'(\alpha)) \geq \lambda(k')+1$, using the  definition of the function $\aleph$ in \eqref{A23}, we obtain
\begin{eqnarray}\nonumber
\aleph(k'(\alpha)) - \aleph(k') [k: k'] &=&\left(2 [k: k'(\alpha)]\right)^{\lambda(k'(\alpha))- \lambda(k)} 
		 - \left(2 [k: k']\right)^{\lambda(k')- \lambda(k)} [k: k']\\ 
  &\geq & 2^{\lambda(k')- \lambda(k)}\left( [k: k']\right)^{\lambda(k')- \lambda(k) +1},
  \end{eqnarray}
where the inequality follows from our assumption  that $\alpha \not \in k'$ and therefore  $$2 [k: k'(\alpha)] \leq [k:k'].$$
\end{proof}

The proof of Lemma \ref{why2} explains the reason why the factor $2$ in the definition of the function $\aleph$ cannot be replaced by a larger number.  Unlike the inequality \eqref{A10} which holds for every $k' \subseteq k$, we could have
\begin{equation*}\label{either-or}
D_{k'} <  D_{k}^{\aleph(k')}    \qquad \textrm{or}\, \qquad  D_{k'} \geq  D_{k}^{\aleph(k')}.
\end{equation*}
In fact, we have
$$
1  = D_{\mathbb{Q}} < D_{k}^{\aleph (\mathbb{Q})}  \, \,   \, \, \textrm{and}\, \, D_k = D_{k}^{\aleph(k)}.
$$
Therefore there exists  a \emph{maximal} field $k^{\ast}  \in \I(k)$, with $k^{\ast} \neq k$ such that 
$$
D_{k^{\ast}} < D_{k}^{\aleph(k^{\ast})},
$$
where by  \emph{maximal} we mean that if $k^{\ast} \subseteq k' \subseteq k$ and $k^{\ast} \neq k'$  then 
$$
D_{k'} \geq  D_{k}^{\aleph(k')}.
$$
It could happen that all proper subfields of $k$ have large enough absolute discriminant so that  $k^{\ast}= \mathbb{Q}$.

 Now having a maximal subfield $k^{\ast}$ of $k$ fixed, we will find independent units $\beta_{1}, \ldots, \beta_{r(k^{\ast})}$ in $k^{\ast}$ and 
 independent relative units   
 $\psi_{1}, \ldots, \psi_{r(k/k^{\ast})}$ in $k$ that satisfy Proposition \ref{lemshort2}. By our choice of $k^{\ast}$, for every 
 $\psi \in \{ \psi_{1}, \ldots, \psi_{r(k/k^{\ast})}\}$, we have
 \begin{equation}\label{ourk}
 D_{k^{\ast}} <  D_{k}^{\aleph(k^{\ast})} \qquad \textrm{and} \qquad D_{k^{\ast}(\psi)} \geq  D_{k}^{\aleph(k^{\ast}(\psi))}.
 \end{equation}
 
 By \eqref{intro61} and \eqref{short440}, and applying \cite[Theorem 1.1]{akhtari2016} to the intermediate  number field $k^{\ast}$, we have
 \begin{equation}\label{boundbetas}
 0.2  < \prod_{i=1}^{r(k^{\ast})} [k^{\ast}:\mathbb{Q}] h(\beta_{i}).
 \end{equation}
 
 Let  $\psi \in \{ \psi_{1}, \ldots, \psi_{r(k/k^{\ast})}\}$. By  \cite[Proposition 4.15]{narkiewicz2010}),
 \begin{equation}\label{DKL}
 D_{k}^{\aleph(k^{\ast}(\psi))} \leq  D_{k^{\ast}(\psi)}
 \leq  D_{k^{\ast}} ^{ [k^{\ast}(\psi): k^{\ast}] } D_{\mathbb{Q}(\psi)}^{ [k^{\ast}(\psi): \mathbb{Q}(\psi)] }\, .
 \end{equation}
 We will consider two cases. First we assume that $k^{\ast} = \mathbb{Q}$. By definition, we have
$$
\aleph(\mathbb{Q}(\psi)) = \left(2 [k: \mathbb{Q}(\psi)]\right)^{\lambda\left(\mathbb{Q}(\psi) \right) - \lambda(k)} \geq d^{1 - \log_{2}d},
$$
where $d = [k:\mathbb{Q}]$.  Therefore, by \eqref{DKL},  we have
\begin{equation}\label{DKLforQ}
 D_{k}^{ d^{1 - \log_{2}d}} \leq D_{k}^{\aleph(\mathbb{Q}(\psi))} \leq  D_{\mathbb{Q}(\psi)}.
\end{equation}
For the second case, assume $k^{\ast} \neq \mathbb{Q}$.  By \eqref{ourk} and \eqref{DKL},  and  since  
 $$
 [k^{\ast}(\psi): \mathbb{Q}(\psi)] \leq \frac{d}{2},
 $$ 
 we have
 \begin{eqnarray}\label{DKLfor2}\nonumber
 D_{k}^{\aleph (k^{\ast}(\psi))} &\leq&  D_{k^{\ast}(\psi)} \leq  D_{k^{\ast}} ^{ [k^{\ast}(\psi): k^{\ast}] } D_{\mathbb{Q}(\psi)}^{ [k^{\ast}(\psi): \mathbb{Q}(\psi)] }\\
 &<&  D_{k} ^{\aleph (k^{\ast}) [k^{\ast}(\psi): k^{\ast}] } D_{\mathbb{Q}(\psi)}^{d/2}\\ \nonumber
 &\leq&  D_{k} ^{\aleph (k^{\ast}) [k: k^{\ast}] } D_{\mathbb{Q}(\psi)}^{d/2} . \end{eqnarray}

 Therefore,
 \begin{equation}\label{DDl}
 \log D_{\mathbb{Q}(\psi)} >  {\frac{2\left(\aleph (k^{\ast}(\psi))-\aleph (k^{\ast}) [k: k^{\ast}]\right)}{d}  } \log D_{k}
 \end{equation}
 Taking $k' = k^{\ast}$ in \eqref{alphadif}, we get
 \begin{eqnarray}\label{alphadif0}
 \aleph(k^{\ast}(\psi)) - \aleph(k^{\ast}) [k: k^{\ast}]   &\geq  & 2^{\lambda(k^{\ast})- \lambda(k)}\left( [k: k^{\ast}]\right)^{\lambda(k^{\ast})- \lambda(k) +1}\\ \nonumber &\geq  & d^{\lambda(k^{\ast})-\lambda(k)} \left(\frac{d}{2}\right),
 \end{eqnarray}
 where the last inequality is a consequence  of our assumption that $k^{\ast} \neq \mathbb{Q}$, and therefore $[k:k^{\ast}] \leq \frac{d}{2}$.
 By \eqref{DDl} and \eqref{alphadif0}, we have
 \begin{equation}\label{DKL2}
 D_{\mathbb{Q}(\psi)} >   D_{k}^{d^{\lambda(k^{\ast})-\lambda(k)} } \geq D_{k}^{d^{1 - \log_{2}d} }.
 \end{equation}
 
 Let $m = [\mathbb{Q}(\psi): \mathbb{Q}]$. By \eqref{Dhm}, we have
   $$
   h(\psi) \geq  \frac{\log \frac{D_{\mathbb{Q}(\psi)}}{m^m}}{2m (m-1)}.
   $$
   This, together with \eqref{DKLforQ} and \eqref{DKL2}, implies that
   $$
   h(\psi) \geq  \frac{\log \frac{D_{k}^{d^{1 - \log_{2}d} }}{m^m}}{2m (m-1)}.
   $$
   If $k^{\ast} \neq \Q$ or $\Q(\psi) \neq k$, we have $m \leq \frac{d}{2}$, and therefore,
\begin{equation}\label{h>Dl}
   h(\psi) \geq  \frac{2\, \log D_{k}^{d^{1 - \log_{2}d} } - d\log{d}}{d (d-2)}.
   \end{equation}
In case $k^{\ast} = \Q$ and $\Q(\psi) = k$, by \eqref{Dhm}, we obtain
   $$
   h(\psi) \geq \frac{\log \frac{D_{k}}{d^d}}{2d (d-1)} > \frac{2\, \log D_{k}^{d^{1 - \log_{2}d} } - d\log{d}}{d (d-2)}.
   $$
Now Theorem \ref{Sil-alpha} follows from Proposition \ref{lemshort2},  and by \eqref{boundbetas}, \eqref{h>Dl}, and noticing via $r(k^{\ast}) \leq \rho(k)$ that 
\begin{equation*}\label{whatif}
r(k/k^{\ast}) \geq r(k) - \rho(k).
\end{equation*}

We conclude by noting that in case $k^{\ast}$ is $\Q$ or a quadratic imaginary extension of $\Q$, we do not need to use Proposition \ref{lemshort2} on existence of multiplicatively independent relative units with small heights. Instead, we can simply 
apply  \eqref{short440} that guarantees the existence of multiplicatively independent (ordinary) units with small heights.

\section*{Acknowledgements} The authors are grateful to the anonymous referee for helpful comments and suggestions.
Shabnam Akhtari's research has been in parts supported by  \emph{the Simons Foundation Collaboration Grants}, Award Number 635880, and by  \emph{the National Science Foundation} Award DMS-2001281.
 

\end{document}